\documentclass[final,3p,times]{elsarticle}

\usepackage[T1]{fontenc}
\usepackage[utf8]{inputenc}
\usepackage{lmodern}
\usepackage{microtype}
\usepackage{amsmath,amssymb,mathtools}
\usepackage{amsthm}
\usepackage{enumitem}
\usepackage{xcolor}
\usepackage[colorlinks=true,linkcolor=blue!45!black,citecolor=blue!45!black,urlcolor=blue!45!black]{hyperref}
\biboptions{sort&compress}

\theoremstyle{plain}
\newtheorem{theorem}{Theorem}[section]
\newtheorem{proposition}[theorem]{Proposition}
\newtheorem{lemma}[theorem]{Lemma}
\newtheorem{corollary}[theorem]{Corollary}

\theoremstyle{definition}
\newtheorem{definition}[theorem]{Definition}
\newtheorem*{definition*}{Definition}

\newtheorem{example}[theorem]{Example}

\newcommand{\R}{\mathbb{R}}

\newcommand{\id}{\operatorname{id}}
\newcommand{\Int}{\operatorname{int}}
\newcommand{\cone}{\operatorname{cone}}
\newcommand{\conv}{\operatorname{conv}}
\newcommand{\supp}{\operatorname{supp}}
\newcommand{\ext}{\operatorname{ext}}
\newcommand{\End}{\operatorname{End}}

\newcommand{\Rc}{\mathcal R_C}

\newcommand{\restr}{\mathbin{\vrule height 1.35ex depth -0.3ex width 0.08ex\vrule height 0.08ex depth -0.02ex width 0.7ex}}
\newcommand{\dd}{\,d}
\newcommand{\one}{\mathbf 1}

\journal{Preprint}

\begin{document}

\begin{frontmatter}

\title{Cone domains separate FS-domains from RB-domains}

\author[addr1]{Yuxu Chen}
\address[addr1]{School of Mathematics, Sichuan University}
\ead{chenyuxu@scu.edu.cn}



\begin{abstract}
Let $C$ be a closed, convex, pointed and generating cone in a finite-dimensional real vector space $V$, and let
\(
        D_C=(-C)\cup\{\bot\}
\)
be the negative cone with a new least element, ordered by the cone order.  Keimel proved that these cone domains are FS-domains and asked whether they are always retracts of bifinite domains.  We give a sharp answer:
\[
        D_C\text{ is an RB-domain}\quad\Longleftrightarrow\quad C\text{ is simplicial}.
\]
Thus every non-simplicial proper cone gives an FS-domain which is not an RB-domain.  The proof converts the RB approximation property into finite-valued $C$-monotone approximations of the identity.  The analytic obstruction is elementary and finite-dimensional: first in Euclidean space, cone-upper sets are represented, up to null sets, as Lipschitz epigraphs; Rademacher's theorem, Fubini's theorem and integration by parts then force the matrix tested against any finite-valued monotone map to lie in the cone generated by the positive rank-one operators $v\otimes\ell$, $v\in C$, $\ell\in C^*$.  If such maps approximate the identity, the identity operator lies in this rank-one cone, which is possible exactly when the cone is simplicial.  This answers Keimel's question in the negative for the Lorentz cone and other non-simplicial cones.
\end{abstract}


\end{frontmatter}

\section{Introduction}

The comparison between FS-domains and retracts of bifinite domains is one of the classical approximation questions in domain theory.  An FS-domain is a domain whose identity map is the directed supremum of Scott-continuous maps finitely separated from the identity.  An RB-domain is a Scott-continuous retract of a bifinite domain, or equivalently a dcpo carrying a directed approximate identity of finite-range deflations.  Bifinite domains and RB-domains provide a robust finite-approximation framework, while FS-domains form a larger-looking class.  Whether these two approximation notions coincide has been a persistent problem since Jung's work on Cartesian closed categories of domains \cite{Jung1989}; see also Abramsky and Jung \cite{AbramskyJung1994}.

Keimel introduced a geometric family of test cases in his study of bicontinuous domains and old problems in domain theory \cite{Keimel2009}.  Let $C$ be a closed, convex, pointed and generating cone in a finite-dimensional real vector space $V$.  The cone induces the order
\[
        x\leq_C y \quad\Longleftrightarrow\quad y-x\in C.
\]
Keimel considered the dcpo
\[
        D_C=(-C)\cup\{\bot\},
\]
where $\bot$ is a new least element.  He stated that $D_C$ is always an FS-domain and asked whether it must be a retract of a bifinite domain \cite[Proposition~5.4 and Problem~5.5]{Keimel2009}.  This question is especially transparent for the three-dimensional ice-cream cone, also known as the Lorentz cone.

We prove that the answer is completely governed by the elementary convex geometry of $C$.

\begin{theorem}\label{thm:intro-main}
Let $C\subseteq V$ be a closed, convex, pointed and generating cone in a finite-dimensional real vector space.  Then
\[
        D_C=(-C)\cup\{\bot\}
\]
is an FS-domain.  Moreover,
\[
        D_C\text{ is an RB-domain}\quad\Longleftrightarrow\quad C\text{ is simplicial}.
\]
\end{theorem}

Here a proper cone is simplicial if it is generated by a basis of $V$.  Hence a non-simplicial cone gives an FS-domain which is not an RB-domain.  In particular, for the Lorentz cone
\[
        L_3=\{(t,x,y)\in\R^3:t\geq (x^2+y^2)^{1/2}\},
\]
the domain $(-L_3)_\bot$ belongs to $\mathbf{FS}\setminus\mathbf{RB}$.

The proof follows the same geometric strategy throughout.  The order-theoretic part first identifies the way-below relation in $D_C$:
\[
        x\ll y\quad\Longleftrightarrow\quad y-x\in\operatorname{int}C
        \qquad(x,y\in -C).
\]
If $D_C$ is assumed to be RB, finite-range deflations produce finite-valued $C$-monotone maps $Q_\varepsilon:P\to C$ on compact subsets $P\subseteq\operatorname{int}C$ satisfying
\[
        c\leq_C Q_\varepsilon(c)\leq_C c+\varepsilon a.
\]
The analytic step is proved first in finite-dimensional Euclidean space.  For a finite-valued monotone map, one slices the level sets by upper layers.  Each upper layer is, up to a null set, a Lipschitz epigraph.  The epigraph integration formula obtained from Rademacher's theorem and Fubini's theorem shows that, for every nonnegative test function $\psi$,
\[
        -\int Q(x)\otimes d\psi_x\,dx
        \in
        \Rc:=\operatorname{cone}\{v\otimes\ell:v\in C,\ \ell\in C^*\}\subseteq \End(V).
\]
If such maps approximate the identity locally uniformly, integration by parts gives $I_V\in\Rc$.  Finally, an elementary convex-geometric argument shows that this identity decomposition forces a compact base of $C$ to be a simplex.  Since every finite-dimensional real vector space is linearly isomorphic to a Euclidean space, the Euclidean obstruction applies to arbitrary finite-dimensional $V$ by changing coordinates.

The paper is organized as follows.  Section~\ref{sec:preliminaries} recalls the domain-theoretic and cone-theoretic notation.  Section~\ref{sec:order-approx} proves the order properties of $D_C$, Keimel's FS assertion, and the finite-valued approximation consequence of the RB hypothesis.  Section~\ref{sec:elementary-obstruction} develops the elementary finite-dimensional obstruction.  Section~\ref{sec:classification} constructs the RB approximation in the simplicial case and derives the main consequences.

\section{Domain-theoretic and cone-theoretic preliminaries}\label{sec:preliminaries}

We use standard domain-theoretic terminology as in \cite{AbramskyJung1994}.  A subset $A$ of a poset is \emph{directed} if it is nonempty and every two elements of $A$ have an upper bound in $A$.  A \emph{dcpo} is a poset in which every directed subset has a supremum.  A map between dcpos is \emph{Scott-continuous} if it is monotone and preserves directed suprema.

For elements $x,y$ of a dcpo $D$, write $x\ll y$ if, whenever $A\subseteq D$ is directed and $y\leq\sup A$, some $a\in A$ satisfies $x\leq a$.  A dcpo $D$ is \emph{continuous} if $\{x:x\ll y\}$ is directed with supremum $y$ for every $y\in D$.  In this paper a \emph{domain} means a continuous dcpo.

Let $D$ be a dcpo.  A Scott-continuous map $f:D\to D$ is \emph{finitely separated from the identity} if there is a finite set $M\subseteq D$ such that for every $x\in D$ one can choose $m\in M$ with
\[
        f(x)\leq m\leq x.
\]

\begin{definition}[FS-domain]\label{def:FS-domain}
A domain $D$ is an \emph{FS-domain} if the identity map $\id_D$ is the pointwise supremum of a directed family of Scott-continuous self-maps finitely separated from the identity.
\end{definition}

A \emph{deflation} on a dcpo $D$ is a Scott-continuous map $r:D\to D$ with finite image and $r\leq\id_D$ pointwise.  A family $(r_i)_{i\in I}$ of deflations is directed when it is directed for the pointwise order on self-maps, and its pointwise supremum is the map $x\mapsto\sup_i r_i(x)$ whenever these suprema exist.

\begin{theorem}[Finite-deflation characterization of RB-domains]\label{thm:Jung-deflation}
A dcpo is an RB-domain if and only if it admits a directed family of deflations whose pointwise supremum is the identity map.  Equivalently, an RB-domain is a dcpo admitting an approximate identity by finite-range Scott-continuous maps below the identity \cite[Theorem~4.1]{Jung1989}.
\end{theorem}

We next recall the cone terminology used below.

\begin{definition}[Cone terminology]\label{def:cone-terminology}
Let $V$ be a finite-dimensional real vector space.  The basic notions are as in  \cite{BoydVandenberghe2004}.  For a subset $S\subseteq V$, \(\cone(S)\) denotes the set of all finite nonnegative linear combinations of elements of $S$, and \(\conv(S)\) denotes the convex hull.

\begin{enumerate}[label=\textup{(\alph*)}]
\item A subset $C\subseteq V$ is a \emph{cone} if
\(
        c\in C, t\geq0
        \quad\Longrightarrow\quad
        tc\in C.
\)
\item A subset $C\subseteq V$ is \emph{convex} if
\(
        x,y\in C, 0\leq\theta\leq1
        \quad\Longrightarrow\quad
        \theta x+(1-\theta)y\in C.
\)

Thus $C$ is a convex cone if and only if
\(
        x,y\in C, s,t\geq0
        \quad\Longrightarrow\quad
        sx+ty\in C.
\)
\item A subset $C\subseteq V$ is \emph{closed} if
\(
        \overline C=C.
\)
\item A cone $C\subseteq V$ is \emph{pointed} if
\(
        C\cap(-C)=\{0\},
\)
and \emph{generating} if
\(
        C-C=V.
\)
For finite-dimensional convex cones, the generating condition is equivalent to $\Int C\neq\varnothing$.
\item In this paper a cone $C\subseteq V$ is called \emph{proper} if it is closed, convex, pointed, and generating.
\item A proper cone $C\subseteq V$ is \emph{simplicial} if there is a basis $v_1,\ldots,v_d$ of $V$ such that
\[
        C=\cone\{v_1,\ldots,v_d\}
        =\left\{\sum_{j=1}^d t_jv_j:t_j\geq0\right\}.
\]
Equivalently, a compact base of $C$ is a simplex.  The same notion is called a simplex cone, or Yudin cone, in \cite[\S1.2]{deBruyn2020}.
\end{enumerate}
\end{definition}

Let $C\subseteq V$ be a cone.  Its \emph{dual cone} is
\(
        C^*=\{\ell\in V^*: \ell(c)\geq 0\text{ for all }c\in C\}.
\)

\begin{definition}[Cone order]\label{def:cone-order}
Let $C\subseteq V$ be a cone.  The \emph{cone order} associated with $C$ is the relation
\[
        x\leq_C y
        \quad\Longleftrightarrow\quad
        y-x\in C.
\]
\end{definition}


\begin{definition}[$C$-monotone maps]\label{def:C-monotone}
Let $A\subseteq V$ and let $Q:A\to V$.  The map $Q$ is \emph{$C$-monotone} if it preserves the cone order:
\[
        x,y\in A,\quad x\leq_C y
        \quad\Longrightarrow\quad
        Q(x)\leq_C Q(y).
\]
Equivalently,
\[
        x,y\in A,\quad y-x\in C
        \quad\Longrightarrow\quad
        Q(y)-Q(x)\in C.
\]

\end{definition}

We shall use the following finite-dimensional analytic conventions.  Lebesgue measure is used throughout.  For vector-valued, covector-valued, or operator-valued functions, measurability and integration are understood componentwise after choosing coordinates.  If $O\subseteq\R^d$ is open, $C_c^\infty(O)$ denotes the smooth real-valued functions on $\R^d$ whose supports are compact subsets of $O$.  The support of a function is denoted by $\supp$.  A statement that two sets agree up to a null set is always meant with respect to the completed Lebesgue measure.

\begin{theorem}[Finite-dimensional bipolar theorem] \cite[\S2.6.1]{BoydVandenberghe2004}\label{thm:bipolar}
Let $K$ be a closed convex cone in a finite-dimensional real vector space.  Then $K^{**}=K$.
\end{theorem}

\begin{lemma}[Standard finite-dimensional cone facts]\label{lem:standard-cone-facts}
Let $C\subseteq V$ be a proper cone.  Then:
\begin{enumerate}[label=\textup{(\roman*)}]
\item $\Int C\neq\varnothing$, and $C+\Int C\subseteq\Int C$.
\item The dual cone $C^*$ is proper, $\Int C^*\neq\varnothing$, and every $\ell\in\Int C^*$ is strictly positive on $C\setminus\{0\}$.
\item For every $\alpha\in\Int C^*$, the slice
\[
        K_\alpha=\{c\in C:\alpha(c)=1\}
\]
is a compact base of $C$.  Here this means that $K_\alpha$ is compact and convex, and every nonzero $c\in C$ has a unique representation
\[
        c=t k,\qquad t>0,\quad k\in K_\alpha .
\]
Moreover, the truncated cone
\[
        T_\alpha=\{c\in C:\alpha(c)\leq 1\}
\]
is compact.
\item The open cones $\Int C$ and $\Int C^*$ contain bases of $V$ and $V^*$, respectively.
\item The bipolar identity $C^{**}=C$ holds.
\end{enumerate}
\end{lemma}

\begin{proof}
The equivalence, for finite-dimensional convex cones, between the generating condition $C-C=V$ and the solid condition $\Int C\neq\varnothing$ is recorded by Keimel \cite[Section~4, Properties~(ii)]{Keimel2009} and is also reflected in Boyd--Vandenberghe's proper-cone convention \cite[\S2.4.1]{BoydVandenberghe2004}.  Keimel's Remark~5.2 gives the compact-base construction \cite[Remark~5.2]{Keimel2009}.  The inclusion $C+\Int C\subseteq\Int C$ follows from the elementary properties of generalized inequalities in \cite[\S2.4.1]{BoydVandenberghe2004}.  The dual-cone assertions follow from the dual-cone properties in \cite[\S2.6.1]{BoydVandenberghe2004}, and the bipolar identity follows from Theorem~\ref{thm:bipolar}.

The compactness of $T_\alpha$ follows from the compact-base statement: every nonzero point of $T_\alpha$ can be written as $t k$ with $0<t\leq1$ and $k\in K_\alpha$, so
\[
        T_\alpha=\{0\}\cup\{t k:0<t\leq1,\ k\in K_\alpha\},
\]
which is the continuous image of the compact set $[0,1]\times K_\alpha$ under $(t,k)\mapsto tk$.

It remains only to spell out the basis assertion.  If $U\subseteq V$ is a nonempty open set, then $U$ contains a basis of $V$: choose $a\in U$ and choose small vectors $v_1,\ldots,v_d$ so that $a+v_1,\ldots,a+v_d$ are linearly independent and still lie in $U$.  Applying this to $U=\Int C$ gives a basis of $V$ contained in $\Int C$, and applying the same argument in $V^*$ to $\Int C^*$ gives a basis of $V^*$ contained in $\Int C^*$.
\end{proof}

Throughout the paper, $V$ has dimension $d\geq 1$, $C\subseteq V$ is proper, and
\[
        D_C=(-C)\cup\{\bot\}
\]
is ordered by the cone order on $-C$ and by declaring $\bot$ to be least.  We fix an auxiliary norm and Euclidean structure when discussing compactness and Lebesgue measure.  The order-theoretic assertions do not depend on these choices.

\section{Order properties of the cone domain and RB local approximation}\label{sec:order-approx}

\subsection{The order structure of the cone domain}

\begin{lemma}[Compact order intervals]\label{lem:compact-intervals}
If $u\leq_C v$ in $V$, then
\(
        [u,v]_C=(u+C)\cap(v-C)
\)
is compact in the Euclidean topology.
\end{lemma}

\begin{proof}
First,
\(
        [u,v]_C=\{z:z-u\in C,\ v-z\in C\}=(u+C)\cap(v-C).
\)
Since $C$ is closed, the translate $u+C$ is closed.  Also $-C$ is closed, hence $v-C=v+(-C)$ is closed.  Therefore $[u,v]_C$ is closed.

It remains to prove boundedness.  If the interval were unbounded, we could choose $x_n\in [u,v]_C$ with $\|x_n\|\to\infty$ and, after passing to a subsequence, $x_n/\|x_n\|\to w$ with $\|w\|=1$.  Since $x_n-u\in C$ and $v-x_n\in C$, division by $\|x_n\|$ gives
\[
        \frac{x_n-u}{\|x_n\|}\in C,
        \qquad
        \frac{v-x_n}{\|x_n\|}\in C.
\]
Passing to the limit and using closedness gives $w\in C$ and $-w\in C$.  This contradicts pointedness, because $C\cap(-C)=\{0\}$ but $\|w\|=1$.  Thus the interval is closed and bounded, hence compact in finite-dimensional Euclidean topology.
\end{proof}

\begin{lemma}[Directed convergence]\label{lem:directed-convergence}
Every directed subset $A\subseteq -C$ has a supremum in $-C$.  Moreover, when $A$ is regarded as a net indexed by its own order, it converges in the Euclidean topology to $\sup A$.
\end{lemma}

\begin{proof}
Choose $a_0\in A$.  The cofinal tail
\[
        A_0=\{a\in A:a_0\leq_C a\}
\]
is contained in the compact interval $[a_0,0]_C$.  Hence the net has a convergent subnet, say to $z\in [a_0,0]_C$.

We show that $z=\sup A$.  For any $a\in A$, directedness gives $b\in A_0$ with $a,a_0\leq_C b$.  The subnet is eventually above $b$, and closedness of the order gives $b\leq_C z$, hence $a\leq_C z$.  Thus $z$ is an upper bound.  If $w$ is any upper bound of $A$, every term of the subnet is below $w$, whence closedness gives $z\leq_C w$.  Thus $z$ is the least upper bound.  The same argument applies to every convergent subnet, so the net has a unique cluster point in a compact Hausdorff space and therefore converges to $z$.
\end{proof}

It follows that $D_C$ is a dcpo.  Indeed, a directed subset of $D_C$ either equals $\{\bot\}$, whose supremum is $\bot$, or has a cofinal non-bottom part contained in $-C$, whose supremum exists by Lemma~\ref{lem:directed-convergence}.  The following formula is the restriction to $-C$, with a bottom added, of the strict-order formula for vector-space cone orders recorded by Keimel \cite[Section 4, Properties (i)]{Keimel2009}.

\begin{lemma}[Way-below relation]\label{lem:waybelow}
For $x,y\in -C$,
\[
        x\ll y\quad\Longleftrightarrow\quad y-x\in\Int C.
\]
Moreover, $\bot\ll z$ for every $z\in D_C$.
\end{lemma}

\begin{proof}
Assume first that $y-x\in\Int C$.  Let $A\subseteq D_C$ be directed, let $z=\sup A$, and suppose $y\leq_C z$.  A cofinal tail of $A$ lies in $-C$ and converges to $z$ by Lemma~\ref{lem:directed-convergence}.  Since
\[
        z-x=(z-y)+(y-x)\in C+\Int C\subseteq \Int C,
\]
using Lemma~\ref{lem:standard-cone-facts}, the Euclidean neighbourhood $x+\Int C$ of $z$ eventually contains an element of $A$.  That element is above $x$, so $x\ll y$.

Conversely, fix $a\in\Int C$.  The sequence
\[
        y_n=y-\frac1n a
\]
is increasing in $-C$ and has supremum $y$.  If $x\ll y$, then $x\leq_C y_n$ for some $n$.  Hence
\[
        y-x=(y-y_n)+(y_n-x)\in \frac1n a+C\subseteq\Int C.
\]
The assertion for $\bot$ is immediate from the definition of $\ll$.
\end{proof}

\begin{corollary}[Continuity of the cone domain]\label{cor:DC-domain}
The dcpo $D_C$ is a continuous dcpo.  Hence $D_C$ is a domain.
\end{corollary}

\begin{proof}
The bottom element is compact because $\bot\ll\bot$ by Lemma~\ref{lem:waybelow}.  Let $y\in -C$ and choose $a\in\Int C$.  For $n\geq1$, put
\[
        y_n=y-\frac1n a.
\]
Then $y_n\in -C$, the sequence $(y_n)$ is increasing for the cone order, and
\[
        y-y_n=\frac1n a\in\Int C.
\]
Thus $y_n\ll y$ by Lemma~\ref{lem:waybelow}.  The element $y$ is an upper bound of the sequence.  If $w$ is another upper bound, then $w-y_n\in C$ for every $n$; passing to the limit and using closedness of $C$ gives $w-y\in C$, so $y\leq_C w$.  Hence $\sup_n y_n=y$.

It remains to check directedness of the whole approximating set below $y$.  If $x_1,x_2\ll y$, then $y-x_1$ and $y-x_2$ lie in $\Int C$.  Since $a/n\to0$ and $\Int C$ is open, for all sufficiently large $n$ one has
\[
        y-\frac1n a-x_i\in\Int C
        \qquad(i=1,2).
\]
Therefore $x_i\leq_C y_n\ll y$ for $i=1,2$.  The way-below elements below $y$ are directed and have supremum $y$.
\end{proof}

\subsection{Keimel's FS property}

Keimel states the following result in \cite[Proposition~5.4]{Keimel2009}.  We include the proof, because the truncation and finite-separation estimates are used to make explicit why the cone domains are FS-domains.

\begin{proposition}\label{prop:FS}
For every proper cone $C$, the domain $D_C$ is an FS-domain.
\end{proposition}

\begin{proof}
By Corollary~\ref{cor:DC-domain}, $D_C$ is a domain.  Choose $\varphi\in\Int C^*$ and $a\in\Int C$ with $\varphi(a)=1$.  For $n\geq 1$, define $f_n:D_C\to D_C$ by $f_n(\bot)=\bot$ and, for $x\in -C$,
\[
        f_n(x)=
        \begin{cases}
        x-\dfrac1n a, & \varphi(x)>-n,\\[1ex]
        \bot, & \varphi(x)\leq -n.
        \end{cases}
\]
The value $x-a/n$ lies in $-C$, because $-x+a/n\in C$.

The map $f_n$ is monotone.  Indeed, suppose $x\leq_C y$.  If $f_n(x)=\bot$, there is nothing to prove.  If $f_n(x)\ne\bot$, then $\varphi(x)>-n$ and hence
\[
        \varphi(y)=\varphi(x)+\varphi(y-x)\geq\varphi(x)>-n,
\]
because $y-x\in C$ and $\varphi\in C^*$.  Thus $f_n(y)\ne\bot$, and
\[
        f_n(y)-f_n(x)=y-x\in C.
\]
Therefore $f_n(x)\leq_C f_n(y)$.

We verify Scott-continuity.  Let $A\subseteq D_C$ be directed with supremum $z$.  If $z=\bot$, then $A=\{\bot\}$ and the assertion is immediate.  Assume $z\in -C$.  If $\varphi(z)\leq -n$, then every non-bottom $x\in A$ satisfies $x\leq_C z$ and hence $\varphi(x)\leq\varphi(z)\leq -n$; so $f_n[A]=\{\bot\}$ and $f_n(z)=\bot$.  If $\varphi(z)>-n$, then the non-bottom cofinal part of $A$ converges to $z$ by Lemma~\ref{lem:directed-convergence}.  Hence a cofinal tail lies in the open half-space $\{x:\varphi(x)>-n\}$.  On this tail $f_n(x)=x-a/n$, and the translated tail converges to $z-a/n$.  It is directed, has supremum $z-a/n$, and is cofinal in $f_n[A]$.  Thus
\[
        f_n(\sup A)=f_n(z)=z-a/n=\sup f_n[A].
\]
So $f_n$ is Scott-continuous.

The sequence $(f_n)$ is increasing pointwise.  If $f_n(x)=\bot$, this is clear.  If $x\in -C$ and $\varphi(x)>-n$, then also $\varphi(x)>-(n+1)$, and
\[
        f_{n+1}(x)-f_n(x)
        =\left(\frac1n-\frac1{n+1}\right)a\in C.
\]
For each non-bottom $x$, the inequality $\varphi(x)>-n$ holds for all sufficiently large $n$, and then $f_n(x)=x-a/n\to x$.  Since $(f_n(x))_n$ is increasing and converges to $x$, its supremum is $x$; for $x=\bot$ the supremum is also $\bot$.  Hence
\[
        \sup_n f_n=\id_{D_C}.
\]

It remains to prove finite separation.  Put
\[
        K_n=\{x\in -C:\varphi(x)\geq -n\}.
\]
The set $K_n$ is compact: under $x\mapsto -x$ it is identified with $\{c\in C:\varphi(c)\leq n\}=n\{c\in C:\varphi(c)\leq1\}$, which is compact by Lemma~\ref{lem:standard-cone-facts}.  For $z\in K_n$, set
\[
        m_z=z-\frac1{2n}a\in -C.
\]
The set
\[
        U_z=\left\{x\in V:
        z-x+\frac1{2n}a\in\Int C,
        \quad
        x-z+\frac1{2n}a\in\Int C
        \right\}
\]
is an open neighbourhood of $z$.  Since $K_n$ is compact, choose $z_1,\ldots,z_N\in K_n$ with
\[
        K_n\subseteq\bigcup_{j=1}^N U_{z_j},
\]
and put
\[
        F_n=\{\bot,m_{z_1},\ldots,m_{z_N}\}\subseteq D_C.
\]
We show that $F_n$ separates $f_n$ from the identity.  If $x=\bot$ or if $x\in -C$ with $\varphi(x)\leq -n$, then $f_n(x)=\bot\leq\bot\leq x$.  If $x\in -C$ and $\varphi(x)>-n$, then $x\in K_n$ and hence $x\in U_{z_j}$ for some $j$.  The defining inequalities for $U_{z_j}$ give
\[
        m_{z_j}-f_n(x)
        =z_j-x+\frac1{2n}a\in C,
        \qquad
        x-m_{z_j}
        =x-z_j+\frac1{2n}a\in C.
\]
Thus
\[
        f_n(x)\leq_C m_{z_j}\leq_C x.
\]
Therefore each $f_n$ is finitely separated from $\id_{D_C}$, and the increasing sequence $(f_n)$ has pointwise supremum $\id_{D_C}$.  By Definition~\ref{def:FS-domain}, $D_C$ is an FS-domain.
\end{proof}

\subsection{The RB local approximation lemma}

The next proposition extracts from the RB hypothesis a family of finite-valued monotone approximants.  This is the order-theoretic bridge to the finite-dimensional cone obstruction.

\begin{proposition}\label{prop:Qapprox}
Assume that $D_C$ is an RB-domain.  Let $a\in\Int C$, let $P\subseteq\Int C$ be compact, and let $\varepsilon>0$.  Then there is a finite-valued $C$-monotone map
\[
        Q_\varepsilon:P\longrightarrow C
\]
such that
\[
        c\leq_C Q_\varepsilon(c)\leq_C c+\varepsilon a
        \qquad(c\in P).
\]
Consequently, for some constant $M_a$ depending only on $a$ and on the chosen norm,
\[
        \sup_{c\in P}\|Q_\varepsilon(c)-c\|\leq M_a\varepsilon.
\]
\end{proposition}

\begin{proof}
By Theorem~\ref{thm:Jung-deflation}, let $(r_i)$ be a directed family of deflations on $D_C$ with $\sup_i r_i=\id_{D_C}$.

For each $c\in P$, define
\[
        u_c=-\left(c+\frac{3\varepsilon}{4}a\right),
        \qquad
        v_c=-\left(c+\frac{\varepsilon}{2}a\right).
\]
Then $v_c-u_c=\varepsilon a/4\in\Int C$, so $u_c\ll v_c$ by Lemma~\ref{lem:waybelow}.  Since $v_c=\sup_i r_i(v_c)$, choose $i_c$ such that
\[
        u_c\leq_C r_{i_c}(v_c).
\]
Set
\[
        N_c=\left\{d\in V:d-c+\frac{\varepsilon}{4}a\in\Int C,
        \quad c-d+\frac{\varepsilon}{2}a\in\Int C\right\}.
\]
Equivalently,
\[
        N_c=\left(c-\frac{\varepsilon}{4}a+\Int C\right)
        \cap\left(c+\frac{\varepsilon}{2}a-\Int C\right).
\]
Thus $N_c$ is open, being an intersection of two translates of open cones.  Also $c\in N_c$, since $(\varepsilon/4)a$ and $(\varepsilon/2)a$ lie in $\Int C$.  Hence the $N_c$ are open neighbourhoods of the points $c\in P$.  By compactness, choose a finite subcover $N_{c_1},\ldots,N_{c_m}$.  Directedness gives an index $i$ such that $r_{i_{c_k}}\leq r_i$ for $k=1,\ldots,m$.

If $d\in P$ and $d\in N_{c_k}$, then the two defining conditions of $N_{c_k}$ give exactly
\[
        u_{c_k}-\bigl(-(d+\varepsilon a)\bigr)
        =d-c_k+\frac{\varepsilon}{4}a\in C,
\]
so $-(d+\varepsilon a)\leq_C u_{c_k}$, and
\[
        (-d)-v_{c_k}=c_k-d+\frac{\varepsilon}{2}a\in C,
\]
so $v_{c_k}\leq_C -d$.  Using monotonicity and $r_i\leq\id$ gives
\[
        -(d+\varepsilon a)
        \leq_C u_{c_k}
        \leq_C r_{i_{c_k}}(v_{c_k})
        \leq_C r_i(v_{c_k})
        \leq_C r_i(-d)
        \leq_C -d.
\]
In particular $r_i(-d)\neq\bot$: otherwise the left inequality would force the genuine non-bottom element $-(d+\varepsilon a)\in -C$ to be below the least element $\bot$.  We may therefore define
\[
        Q_\varepsilon(d)=-r_i(-d).
\]
This map has finite image, because $r_i$ has finite image.  Negating the last chain reverses the cone order and gives
\[
        d\leq_C Q_\varepsilon(d)\leq_C d+\varepsilon a.
\]
If $d\leq_C e$, then $-e\leq_C -d$, so $r_i(-e)\leq_C r_i(-d)$ and hence $Q_\varepsilon(d)\leq_C Q_\varepsilon(e)$.  Thus $Q_\varepsilon$ is $C$-monotone.

Finally, the inequalities $d\leq_C Q_\varepsilon(d)\leq_C d+\varepsilon a$ say precisely that
\[
        Q_\varepsilon(d)-d\in C,
        \qquad
        \varepsilon a-(Q_\varepsilon(d)-d)\in C.
\]
Hence $Q_\varepsilon(d)-d\in[0,\varepsilon a]_C=\varepsilon[0,a]_C$.  The interval $[0,a]_C$ is compact by Lemma~\ref{lem:compact-intervals}.  Taking $M_a=\max\{\|z\|:z\in[0,a]_C\}$ proves the estimate.
\end{proof}

\section{An elementary finite-dimensional obstruction}\label{sec:elementary-obstruction}

The analytic part of the proof is independent of domain theory.  We first prove it in the Euclidean space $\R^d$.  At the end of the section we pass from Euclidean spaces to arbitrary finite-dimensional real vector spaces by a linear change of coordinates.

Let $K\subseteq\R^d$ be a proper cone.  Its dual cone is
\[
        K^*=\{\xi\in(\R^d)^*: \xi(x)\geq0\text{ for all }x\in K\}.
\]
Define the rank-one cone
\[
        \mathcal R_K=
        \cone\{v\otimes\xi:v\in K,\ \xi\in K^*\}
        \subseteq\End(\R^d),
\]
where $(v\otimes\xi)(x)=\xi(x)v$.

\subsection{The rank-one cone}

\begin{lemma}[Closedness of the rank-one cone]\label{lem:RKclosed}
The cone $\mathcal R_K$ is closed in $\End(\R^d)$.
\end{lemma}

\begin{proof}
Choose $e\in\Int K$ and $\alpha\in\Int K^*$.  The normalized bases
\[
        B_K=\{v\in K:\alpha(v)=1\},
        \qquad
        B_{K^*}=\{\xi\in K^*:\xi(e)=1\}
\]
are compact.  For $B_K$, closedness is immediate.  If $B_K$ were unbounded, a normalized unbounded sequence would have a nonzero limit $u\in K$ with $\alpha(u)=0$, contradicting the strict positivity of $\alpha$ on $K\setminus\{0\}$.  The proof for $B_{K^*}$ is the same, using the fact that every nonzero $\xi\in K^*$ satisfies $\xi(e)>0$ when $e\in\Int K$: if $\xi(e)=0$ and $\xi\ne0$, choose $y$ with $\xi(y)>0$; then $e-ty\in K$ for all sufficiently small $t>0$, contradicting $0\leq\xi(e-ty)=-t\xi(y)$.

Let
\[
        S=\{v\otimes\xi:v\in B_K,
        \xi\in B_{K^*}\}.
\]
The set $S$ is compact, and so is $C_0=\conv S$.  Every nonzero generator $v\otimes\xi$ of $\mathcal R_K$ is a positive scalar multiple of an element of $S$: if $v\ne0$ and $\xi\ne0$, then $\alpha(v)>0$ and $\xi(e)>0$, and
\[
        v\otimes\xi
        =\alpha(v)\xi(e)
        \left(\frac{v}{\alpha(v)}\otimes\frac{\xi}{\xi(e)}\right).
\]
Thus
\[
        \mathcal R_K=\{tA:t\geq0,
        A\in C_0\}.
\]
Define $\Phi\in\End(\R^d)^*$ by $\Phi(T)=\alpha(Te)$.  For every $A\in C_0$, one has $\Phi(A)=1$.  Suppose $t_nA_n\to T$, where $t_n\ge0$ and $A_n\in C_0$.  Then
\[
        t_n=\Phi(t_nA_n)\longrightarrow\Phi(T).
\]
If $\Phi(T)=0$, then $t_n\to0$ and boundedness of $C_0$ gives $T=0\in\mathcal R_K$.  If $\Phi(T)>0$, then compactness of $C_0$ gives a subsequence $A_{n_j}\to A\in C_0$, and hence $T=\Phi(T)A\in\mathcal R_K$.  Therefore $\mathcal R_K$ is closed.
\end{proof}

\begin{lemma}[Identity criterion]\label{lem:identity-simplicial-euclidean}
For a proper cone $K\subseteq\R^d$,
\[
        I_d\in\mathcal R_K
        \quad\Longleftrightarrow\quad
        K\text{ is simplicial}.
\]
Consequently, if $K$ is not simplicial, then
\[
        I_d\notin\mathcal R_K.
\]
\end{lemma}

\begin{proof}
Assume first that $I_d\in\mathcal R_K$.  Then
\[
        I_d=\sum_{j=1}^m v_j\otimes\xi_j,
        \qquad
        v_j\in K,
        \quad
        \xi_j\in K^*,
\]
where zero summands have been removed.  Choose $\alpha\in\Int K^*$ and let
\[
        B=\{x\in K:\alpha(x)=1\}
\]
be the compact base of $K$.  Since $v_j\ne0$ and $\alpha$ is strictly positive on $K\setminus\{0\}$, the following normalization is valid:
\[
        p_j=\frac{v_j}{\alpha(v_j)}\in B,
        \qquad
        \eta_j=\alpha(v_j)\xi_j\in K^*.
\]
For $x\in B$,
\[
        x=I_dx=\sum_{j=1}^m\eta_j(x)p_j,
        \qquad
        \eta_j(x)\geq0,
        \qquad
        \sum_{j=1}^m\eta_j(x)=\alpha(x)=1.
\]
Thus every $x\in B$ is a convex combination of $p_1,\ldots,p_m$.  The reverse inclusion $\conv\{p_1,\ldots,p_m\}\subseteq B$ follows from convexity of $B$, so
\[
        B=\conv\{p_1,\ldots,p_m\}.
\]
Let $E=\ext B$.  Since $B$ is a polytope, $E$ is finite, each element of $E$ is one of the points $p_j$, and $B=\conv E$.  If $q\in E$ and $\eta_j(q)>0$, the displayed convex representation of the extreme point $q$ forces $p_j=q$; otherwise $q$ would be a nontrivial convex combination involving another point of $B$.

For $q\in E$, define
\[
        \beta_q=\sum_{\{j:p_j=q\}}\eta_j\in K^*.
\]
If $q,r\in E$, then the preceding observation gives
\[
        \beta_q(r)=
        \begin{cases}
        1,&q=r,\\
        0,&q\ne r.
        \end{cases}
\]
Indeed, for $r$ fixed, only the indices with $p_j=r$ can have $\eta_j(r)>0$, and their total coefficient is $1$.  Hence the points of $E$ are linearly independent: if $\sum_{r\in E}\lambda_r r=0$, applying $\beta_q$ gives $\lambda_q=0$ for every $q\in E$.  Since $K=\cone B=\cone E$ and $K-K=\R^d$, the set $E$ spans $\R^d$.  Therefore $E$ is a basis of $\R^d$, and $K$ is simplicial.

Conversely, suppose $K=\cone\{b_1,\ldots,b_d\}$ for a basis $b_1,\ldots,b_d$ of $\R^d$.  Let $\beta_1,\ldots,\beta_d$ be the dual basis.  If $x=\sum_i t_i b_i\in K$ with $t_i\ge0$, then $\beta_i(x)=t_i\ge0$, so $\beta_i\in K^*$ for each $i$.  The usual basis expansion gives
\[
        I_d=\sum_{i=1}^d b_i\otimes\beta_i\in\mathcal R_K.
\]
\end{proof}

\begin{lemma}[Cone-valued integration]\label{lem:cone-valued-integral}
Let $M$ be a finite-dimensional real vector space, let $A\subseteq M$ be a closed convex cone, let $U\subseteq\R^n$ be measurable, and let $F:U\to M$ be integrable.  If $F(x)\in A$ for almost every $x\in U$, then
\[
        \int_U F(x)\dd x\in A.
\]
\end{lemma}

\begin{proof}
Put $m=\int_U F(x)\dd x$.  Suppose that $m\notin A$.  Since $A$ is a closed convex cone, the finite-dimensional separation theorem gives a linear functional $\Lambda\in M^*$ such that
\[
        \Lambda(m)<0,
        \qquad
        \Lambda(a)\geq0\quad(a\in A).
\]
Equivalently, $\Lambda$ belongs to the dual cone $A^*$ but is negative on $m$.  Since $F(x)\in A$ almost everywhere, $\Lambda(F(x))\ge0$ almost everywhere.  By linearity of finite-dimensional vector-valued integration,
\[
        \Lambda(m)
        =\Lambda\left(\int_U F(x)\dd x\right)
        =\int_U\Lambda(F(x))\dd x
        \geq0,
\]
contradicting $\Lambda(m)<0$.  Hence $m\in A$.
\end{proof}

\subsection{Lipschitz epigraphs for upper sets}

Fix $e\in\Int K$ and choose $\eta\in\Int K^*$ with $\eta(e)=1$.  Put
\[
        W=\ker\eta,
        \qquad
        \R^d=\R e\oplus W.
\]
Thus every $x\in\R^d$ is uniquely written as
\[
        x=te+z,
        \qquad
        t=\eta(x),
        \quad
        z\in W.
\]
Lebesgue measure in $\R^d$ is a fixed positive multiple $J\,\dd t\dd z$ of product measure under this splitting.

Let $a<b$, let $U\subseteq W$ be open, and put
\[
        Z=(a,b)e+U.
\]
A set $E\subseteq Z$ is called $K$-upper in $Z$ if
\[
        x\in E,
        \quad y\in Z,
        \quad y-x\in K
        \quad\Longrightarrow\quad
        y\in E.
\]

\begin{lemma}[Upper sets are Lipschitz epigraphs]\label{lem:upper-lipschitz-epigraph}
Let $E\subseteq Z$ be $K$-upper in $Z$.  Then there is a Lipschitz function
\[
        f_E:U\to[a,b]
\]
such that, with
\[
        G_{f_E}=\{te+z:z\in U,
        f_E(z)<t<b\},
\]
the set $E$ is Lebesgue measurable and $E\triangle G_{f_E}$ has Lebesgue measure zero in $Z$.  Moreover, for every $z\in U$ and every $t\in(a,b)$ with $t\ne f_E(z)$,
\[
        te+z\in E
        \quad\Longleftrightarrow\quad
        t>f_E(z).
\]
\end{lemma}

\begin{proof}
For $z\in U$, set
\[
        E_z=\{t\in(a,b):te+z\in E\}.
\]
Because $e\in K$, the set $E_z$ is an upper subset of the interval $(a,b)$: if $t\in E_z$ and $t<t'<b$, then $(t'e+z)-(te+z)=(t'-t)e\in K$, and upperness gives $t'\in E_z$.

Define
\[
        f_E(z)=\inf E_z,
\]
with the convention $f_E(z)=b$ when $E_z=\varnothing$.  If $E_z=(a,b)$, this gives $f_E(z)=a$.  The preceding upper-set property implies the exact section statement
\[
        t<f_E(z)\Longrightarrow te+z\notin E,
        \qquad
        t>f_E(z)\Longrightarrow te+z\in E.
\]
Thus $E_z$ differs from $(f_E(z),b)$ by at most the single threshold point $f_E(z)$.

We next prove that $f_E$ is Lipschitz.  Since $e\in\Int K$, there is $\rho>0$ such that $e+u\in K$ whenever $\|u\|<\rho$.  Choose $L>1/\rho$.  Then
\[
        L\|w\|e+w\in K
        \qquad(w\in W).
\]
Indeed, this is immediate for $w=0$, and for $w\ne0$ it follows by writing
\[
        L\|w\|e+w=L\|w\|\left(e+\frac{w}{L\|w\|}\right).
\]

Let $z,z'\in U$ and put $w=z'-z$.  Suppose first that
\[
        f_E(z)+L\|w\|<b.
\]
Choose $t$ with $f_E(z)<t<b-L\|w\|$.  Then $te+z\in E$, and
\[
        (t+L\|w\|)e+z'-(te+z)=L\|w\|e+w\in K.
\]
The point $(t+L\|w\|)e+z'$ lies in $Z$, so upperness gives $(t+L\|w\|)e+z'\in E$.  Hence
\[
        f_E(z')\leq t+L\|w\|.
\]
Letting $t\downarrow f_E(z)$ gives
\[
        f_E(z')\leq f_E(z)+L\|z'-z\|.
\]
If $f_E(z)+L\|z'-z\|\ge b$, the same inequality follows from the trivial bound $f_E(z')\le b$.  Interchanging $z$ and $z'$ yields
\[
        |f_E(z')-f_E(z)|\leq L\|z'-z\|,
\]
so $f_E$ is Lipschitz.

Since $f_E$ is continuous, $G_{f_E}$ is Borel.  The exact section statement gives
\[
        E\triangle G_{f_E}
        \subseteq
        \{f_E(z)e+z:z\in U\}\cap Z.
\]
The set on the right is the image of the graph of the Lipschitz function $f_E$ under the linear map $(t,z)\mapsto te+z$, hence it has $d$-dimensional Lebesgue measure zero.  Lebesgue measure is complete, so every subset of this null graph is measurable and null.  Consequently $E$ is measurable and differs from $G_{f_E}$ by a null set.
\end{proof}

\begin{lemma}[Epigraph integration formula]\label{lem:epigraph-integration}
Let $f:U\to[a,b]$ be Lipschitz and put
\[
        G_f=\{te+z:z\in U,
        f(z)<t<b\}.
\]
For $\psi\in C_c^\infty(Z)$, extended by zero outside $Z$, one has the covector identity
\[
        -\int_{G_f} d\psi_x\dd x
        =J\int_U \psi(f(z)e+z)(\eta-df_z)\dd z.
\]
Here $df_z$, originally a covector on $W$, is used at differentiability points of $f$ and is extended to $\R^d$ by setting $df_z(e)=0$; on the null set where $f$ is not differentiable it may be chosen arbitrarily.  Moreover, if $f=f_E$ is obtained from a $K$-upper set $E\subseteq Z$ as in Lemma~\ref{lem:upper-lipschitz-epigraph}, then
\[
        \eta-df_z\in K^*
\]
for almost every $z\in U$ with $f(z)\in(a,b)$.
\end{lemma}

\begin{proof}
By Rademacher's theorem \cite[Section~3.1.2]{EvansGariepy2015}, $f$ is differentiable almost everywhere.  Let $h\in\R^d$ and write it uniquely as
\[
        h=\tau e+w,
        \qquad \tau\in\R,
        \quad w\in W.
\]
Then $\tau=\eta(h)$.  Since $\dd x=J\,\dd t\dd z$, Fubini's theorem gives
\[
\begin{aligned}
 -\int_{G_f}d\psi_x(h)\dd x
 &=-J\int_U\int_{f(z)}^b
        \bigl(\tau\partial_t\psi(te+z)+d_z\psi(te+z)(w)\bigr)
        \dd t\dd z .
\end{aligned}
\]
The $t$-derivative term is
\[
\begin{aligned}
-J\tau\int_U\int_{f(z)}^b\partial_t\psi(te+z)\dd t\dd z
&=-J\tau\int_U\bigl(\psi(be+z)-\psi(f(z)e+z)\bigr)\dd z \\
&=J\tau\int_U\psi(f(z)e+z)\dd z,
\end{aligned}
\]
because $\psi$ has compact support in $Z$ and hence zero trace at the boundary level $t=b$.

For the $W$-term, set
\[
        H(z)=\int_{f(z)}^b\psi(te+z)\dd t.
\]
The function $H$ is Lipschitz on $U$ and has support compactly contained in $U$.  At almost every point where $f$ is differentiable, the Leibniz rule gives
\[
        dH_z(w)=\int_{f(z)}^b d_z\psi(te+z)(w)\dd t
        -\psi(f(z)e+z)df_z(w).
\]
The zero extension of $H$ to $W$ is compactly supported and Lipschitz.  Applying the one-dimensional fundamental theorem along lines in the direction $w$ gives
\[
        \int_U dH_z(w)\dd z=0.
\]
Therefore
\[
        \int_U\int_{f(z)}^b d_z\psi(te+z)(w)\dd t\dd z
        =\int_U\psi(f(z)e+z)df_z(w)\dd z.
\]
Combining the two terms yields
\[
        -\int_{G_f}d\psi_x(h)\dd x
        =J\int_U\psi(f(z)e+z)\bigl(\tau-df_z(w)\bigr)\dd z.
\]
Since $(\eta-df_z)(h)=\tau-df_z(w)$ by the chosen extension of $df_z$, and since $h$ was arbitrary, the covector identity follows.

Now assume $f=f_E$ comes from a $K$-upper set.  Let $z\in U$ be a point where $f$ is differentiable and $f(z)\in(a,b)$.  Take $h=\tau e+w\in K$, with $w\in W$.  Since $\eta\in K^*$,
\[
        \tau=\eta(h)\ge0.
\]
For $s>0$ sufficiently small, $z+sw\in U$ and $f(z)+s\tau<b$.  If $\delta>0$ is chosen so that $f(z)+\delta+s\tau<b$, then the exact section statement in Lemma~\ref{lem:upper-lipschitz-epigraph} gives
\[
        (f(z)+\delta)e+z\in E.
\]
Since $s h=s\tau e+sw\in K$, upperness gives
\[
        (f(z)+\delta+s\tau)e+(z+sw)\in E.
\]
Therefore
\[
        f(z+sw)\leq f(z)+\delta+s\tau.
\]
Letting $\delta\downarrow0$, dividing by $s>0$, and then letting $s\downarrow0$, gives
\[
        df_z(w)\leq\tau.
\]
Thus
\[
        (\eta-df_z)(h)=\tau-df_z(w)\ge0.
\]
Since $h\in K$ was arbitrary, $\eta-df_z\in K^*$.  This holds at every differentiability point with $f(z)\in(a,b)$, hence almost everywhere on that set.
\end{proof}

\subsection{Finite-valued monotone maps}

\begin{lemma}[Finite-valued monotone maps give rank-one-cone fluxes]\label{lem:finite-step-flux}
Let $\Omega\subseteq\R^d$ be open, let $Q:\Omega\to\R^d$ be finite-valued and $K$-monotone, and let
\[
        Z=(a,b)e+U
\]
be a cylinder with $\overline Z\subseteq\Omega$, where $U\subseteq W$ is bounded and open.  If $0\leq\psi\in C_c^\infty(Z)$, then $Q$ is measurable on $Z$ and
\[
        T_Q(\psi):=-\int_Z Q(x)\otimes d\psi_x\dd x
        \in\mathcal R_K.
\]
\end{lemma}

\begin{proof}
Let $q_1,\ldots,q_N$ be the distinct values of $Q$ on $\Omega$.  If $N=1$, then $Q$ is constant and
\[
        T_Q(\psi)=-q_1\otimes\int_Z d\psi_x\dd x=0\in\mathcal R_K.
\]
Assume $N\geq2$.  Choose $\lambda\in\Int K^*$ such that the numbers $\lambda(q_i)$ are pairwise distinct.  This is possible because, for $q_i\ne q_j$, the condition $\lambda(q_i)=\lambda(q_j)$ defines a proper hyperplane in $(\R^d)^*$, and finitely many proper hyperplanes cannot cover the nonempty open set $\Int K^*$.  Relabel the values so that
\[
        \lambda(q_1)<\lambda(q_2)<\cdots<\lambda(q_N).
\]
Choose numbers $s_i$ satisfying
\[
        \lambda(q_i)<s_i<\lambda(q_{i+1}),
        \qquad 1\leq i<N,
\]
and define
\[
        E_i=\{x\in\Omega:\lambda(Q(x))>s_i\}.
\]
Each $E_i$ is $K$-upper in $\Omega$: if $x\in E_i$, $y\in\Omega$, and $y-x\in K$, then $Q(y)-Q(x)\in K$, so
\[
        \lambda(Q(y))\geq\lambda(Q(x))>s_i.
\]
Thus $E_i\cap Z$ is $K$-upper in $Z$.  By Lemma~\ref{lem:upper-lipschitz-epigraph}, let $f_i=f_{E_i\cap Z}:U\to[a,b]$.  Then $E_i\cap Z$ differs from
\[
        G_i=\{te+z:z\in U,
        f_i(z)<t<b\}
\]
by a null set, and for every fixed $z\in U$ and every $t\ne f_i(z)$,
\[
        te+z\in E_i
        \quad\Longleftrightarrow\quad
        t>f_i(z).
\]
Since
\[
        E_1\supseteq E_2\supseteq\cdots\supseteq E_{N-1},
\]
the threshold functions satisfy
\[
        f_1\leq f_2\leq\cdots\leq f_{N-1}.
\]
The sets $E_i\cap Z$ are measurable.  Hence the phase sets of $Q$ in $Z$ are measurable; with $E_0=\Omega$ and $E_N=\varnothing$,
\[
        \{Q=q_j\}\cap Z=(E_{j-1}\cap Z)\setminus(E_j\cap Z).
\]
Thus $Q$ is measurable on $Z$.  The layer formula holds pointwise on $Z$:
\[
        Q=q_1+\sum_{i=1}^{N-1}(q_{i+1}-q_i)\one_{E_i}.
\]
The constant term gives no contribution because $\int_Z d\psi_x\dd x=0$.  Replacing $E_i\cap Z$ by the null-equivalent epigraph $G_i$ and applying Lemma~\ref{lem:epigraph-integration} gives
\[
        T_Q(\psi)
        =J\int_U
        \sum_{i=1}^{N-1}
        \psi(f_i(z)e+z)(q_{i+1}-q_i)\otimes(\eta-df_{i,z})
        \dd z.
\]
It remains to show that the integrand belongs to $\mathcal R_K$ for almost every $z$.

Fix a point $z$ at which all $f_i$ are differentiable.  Partition $\{1,\ldots,N-1\}$ into maximal consecutive blocks on which the numbers $f_i(z)$ are equal.  Let $r,\ldots,s$ be such a block and write
\[
        f_r(z)=\cdots=f_s(z)=t.
\]
If $\psi(te+z)=0$, the block contributes zero.  Suppose $\psi(te+z)>0$.  Since $\supp\psi\subseteq Z$, this implies $t\in(a,b)$.  For $r\leq i<s$, the function $f_{i+1}-f_i$ is nonnegative and has value $0$ at $z$; differentiability at $z$ forces
\[
        df_{r,z}=df_{r+1,z}=\cdots=df_{s,z}.
\]
Therefore the block contribution is
\[
        \psi(te+z)(q_{s+1}-q_r)\otimes(\eta-df_{r,z}).
\]
By Lemma~\ref{lem:epigraph-integration}, $\eta-df_{r,z}\in K^*$.

We claim that $q_{s+1}-q_r\in K$.  By maximality of the block, choose $t_-<t<t_+$ so close to $t$ that
\[
        t_- > f_{r-1}(z)\quad(r>1),
        \qquad
        t_+ < f_{s+1}(z)\quad(s<N-1),
\]
with the evident omissions when $r=1$ or $s=N-1$.  The choices also avoid all threshold levels $f_i(z)$.  The exact section statement above gives the phase values
\[
        Q(t_-e+z)=q_r,
        \qquad
        Q(t_+e+z)=q_{s+1}.
\]
Since $(t_+-t_-)e\in K$ and $Q$ is $K$-monotone,
\[
        q_{s+1}-q_r=Q(t_+e+z)-Q(t_-e+z)\in K.
\]
Thus each nonzero block contribution is a nonnegative scalar multiple of a generator $v\otimes\xi$ with $v\in K$ and $\xi\in K^*$.  Therefore the whole integrand lies in $\mathcal R_K$ for almost every $z$.  Lemma~\ref{lem:cone-valued-integral} now gives $T_Q(\psi)\in\mathcal R_K$.
\end{proof}

\begin{proposition}[Euclidean finite-step obstruction]\label{prop:euclidean-obstruction}
Let $K\subseteq\R^d$ be a proper cone which is not simplicial.  For every nonempty open set $\Omega\subseteq\R^d$, there is no sequence of finite-valued $K$-monotone maps
\[
        Q_n:\Omega\to\R^d
\]
which converges to the identity uniformly on compact subsets of $\Omega$.
\end{proposition}

\begin{proof}
Suppose, to the contrary, that such a sequence $(Q_n)$ exists.  By Lemma~\ref{lem:identity-simplicial-euclidean},
\[
        I_d\notin\mathcal R_K.
\]
Because $\Omega$ is nonempty and open, choose a cylinder
\[
        Z=(a,b)e+U
\]
with $U\subseteq W$ bounded and open and $\overline Z\subseteq\Omega$.  Choose $0\leq\psi\in C_c^\infty(Z)$ with
\[
        \int_Z\psi(x)\dd x=1.
\]
For every $n$, Lemma~\ref{lem:finite-step-flux} gives
\[
        A_n:=-\int_Z Q_n(x)\otimes d\psi_x\dd x\in\mathcal R_K.
\]
Uniform convergence on the compact set $\supp\psi$ implies
\[
\begin{aligned}
\left\|A_n+\int_Z x\otimes d\psi_x\dd x\right\|
&=\left\|\int_Z (x-Q_n(x))\otimes d\psi_x\dd x\right\|  \\
&\leq
\sup_{x\in\supp\psi}\|Q_n(x)-x\|
\int_Z\|d\psi_x\|\dd x
\longrightarrow0.
\end{aligned}
\]
Since $\mathcal R_K$ is closed by Lemma~\ref{lem:RKclosed}, it follows that
\[
        -\int_Z x\otimes d\psi_x\dd x\in\mathcal R_K.
\]
For every $h\in\R^d$, coordinatewise integration by parts gives
\[
\begin{aligned}
        \left(-\int_Z x\otimes d\psi_x\dd x\right)(h)
        &=-\int_Z x\,d\psi_x(h)\dd x  \\
        &=h\int_Z\psi(x)\dd x
        =h.
\end{aligned}
\]
There is no boundary term because $\psi$ has compact support in $Z$.  Hence
\[
        -\int_Z x\otimes d\psi_x\dd x=I_d,
\]
contradicting $I_d\notin\mathcal R_K$.
\end{proof}

\begin{corollary}[Coordinate-free form]\label{cor:vector-space-obstruction}
Let $V$ be a finite-dimensional real vector space and let $C\subseteq V$ be a non-simplicial proper cone.  For every nonempty open set $\Omega\subseteq V$, there is no sequence of finite-valued $C$-monotone maps
\[
        Q_n:\Omega\to V
\]
which converges to the identity uniformly on compact subsets of $\Omega$.
\end{corollary}

\begin{proof}
Choose a linear isomorphism $A:V\to\R^d$ and put $K=A(C)$.  Then $K$ is a non-simplicial proper cone.  If such maps $Q_n$ existed on $\Omega$, then
\[
        \widetilde Q_n(y)=A Q_n(A^{-1}y),
        \qquad y\in A\Omega,
\]
would be finite-valued $K$-monotone maps on the nonempty open set $A\Omega\subseteq\R^d$.  Uniform convergence on compact subsets is preserved under the linear isomorphism $A$, so $\widetilde Q_n\to\id$ locally uniformly on $A\Omega$.  This contradicts Proposition~\ref{prop:euclidean-obstruction}.
\end{proof}

\section{The simplicial classification and consequences}\label{sec:classification}

\subsection{The simplicial case}

It remains to show that simplicial cones really give RB-domains.  We do this directly by constructing finite-range deflations on the model domain $(-\R_+^d)_\bot$ and then applying Jung's theorem.

\begin{lemma}\label{lem:simplicial-RB-model}
The dcpo $(-\R_+^d)_\bot$, ordered coordinatewise on $-\R_+^d$ and with a new least element adjoined, is an RB-domain.
\end{lemma}

\begin{proof}
Let $D=(-\R_+^d)_\bot$.  For $n\geq1$, put $\delta_n=2^{-n}$.  For $t\in(-\infty,0]$ with $t>-n$, define
\[
        s_n(t)=\max\{k\delta_n:k\in\mathbb Z,
        -n\leq k\delta_n<t\}.
\]
The set in the maximum is nonempty because it contains $-n=(-n2^n)\delta_n$, and it is finite because all its elements lie in $[-n,0)$.  Equivalently, $s_n(t)$ is the largest point of the finite grid
\[
        \{-n,-n+\delta_n,\ldots,-\delta_n\}
\]
which is strictly smaller than $t$.  Thus $s_n(t)\le t$, and $s_n$ is monotone on $(-n,0]$.

Define $r_n:D\to D$ by $r_n(\bot)=\bot$ and, for $x=(x_1,\ldots,x_d)\in-\R_+^d$,
\[
        r_n(x)=
        \begin{cases}
        (s_n(x_1),\ldots,s_n(x_d)),& x_i>-n\text{ for all }i,\\[0.6ex]
        \bot,& x_i\leq -n\text{ for some }i.
        \end{cases}
\]
Each $r_n$ has finite image and satisfies $r_n\leq\id_D$.  It is monotone because the threshold condition is upward closed and each scalar map $s_n$ is monotone.

We check Scott-continuity.  Let $A\subseteq D$ be directed with supremum $z$.  If $z=\bot$, then $A=\{\bot\}$.  Suppose $z\in-\R_+^d$.  If some coordinate $z_i\le -n$, then every non-bottom element of $A$ has $i$-th coordinate at most $-n$, so $r_n[A]=\{\bot\}$ and $r_n(z)=\bot$.  If every coordinate of $z$ is $>-n$, then a cofinal tail of $A$ lies in the open box
\[
        \{x\in-\R_+^d:x_i>-n\text{ for all }i\}.
\]
On that tail $r_n$ is the coordinatewise map induced by $s_n$.  The scalar map $s_n$ preserves increasing suprema: it is constant on each interval $(k\delta_n,(k+1)\delta_n]$ with value $k\delta_n$, including the lower value at the grid breakpoint.  Hence the coordinatewise supremum of $r_n[A]$ is $r_n(z)$.  Therefore $r_n$ is Scott-continuous.

The sequence $(r_n)$ is increasing pointwise.  The threshold $-n$ moves downward as $n$ increases, and the grid of mesh $\delta_{n+1}$ refines the grid of mesh $\delta_n$; hence the largest grid point strictly below a fixed active coordinate cannot decrease.  Finally, for every $x\in-\R_+^d$, all coordinates are above $-n$ for large $n$, and the mesh $\delta_n$ tends to $0$.  Therefore
\[
        \sup_n r_n(x)=x,
        \qquad x\in-\R_+^d,
\]
and also $\sup_n r_n(\bot)=\bot$.  Thus $(r_n)$ is a directed family of deflations with pointwise supremum $\id_D$.  By Theorem~\ref{thm:Jung-deflation}, $D$ is an RB-domain.
\end{proof}

\begin{proposition}\label{prop:simplicial-RB}
If $C$ is simplicial, then $D_C$ is an RB-domain.
\end{proposition}

\begin{proof}
Write $C=\cone\{v_1,\ldots,v_d\}$ for a basis $v_1,\ldots,v_d$ of $V$.  Define the linear isomorphism $A:V\to\R^d$ by
\[
        A\left(\sum_{i=1}^d t_i v_i\right)=(t_1,\ldots,t_d).
\]
Then $A(C)=\R_+^d$, and
\[
        x\leq_C y
        \quad\Longleftrightarrow\quad
        A x\leq_{\R_+^d} A y.
\]
Thus $A$ restricts to an order isomorphism $-C\cong-\R_+^d$ and extends to an order isomorphism
\[
        D_C\cong(-\R_+^d)_\bot
\]
by sending $\bot$ to $\bot$.  RB-domains are invariant under order isomorphism, so the result follows from Lemma~\ref{lem:simplicial-RB-model}.
\end{proof}

\subsection{Proof of the main theorem and consequences}

\begin{proof}[Proof of Theorem~\ref{thm:intro-main}]
The FS assertion is Proposition~\ref{prop:FS}.  Suppose that $D_C$ is an RB-domain.  We prove that $C$ is simplicial.

Assume, for a contradiction, that $C$ is not simplicial.  Since $\Int C$ is a nonempty open subset of $V$, we may choose a compact full-dimensional parallelotope
\[
        P\subseteq\Int C
\]
with nonempty interior $\Omega=\Int P$.  Fix $a\in\Int C$ and choose any sequence $\varepsilon_k\downarrow0$.  Proposition~\ref{prop:Qapprox} gives finite-valued $C$-monotone maps
\[
        Q_k:P\to C
\]
such that
\[
        c\leq_C Q_k(c)\leq_C c+\varepsilon_k a
        \qquad(c\in P).
\]
The norm estimate in Proposition~\ref{prop:Qapprox} yields
\[
        \sup_{c\in P}\|Q_k(c)-c\|\leq M_a\varepsilon_k\longrightarrow0.
\]
Therefore the restrictions $Q_k\restr\Omega:\Omega\to V$ are finite-valued $C$-monotone maps converging to the identity uniformly on compact subsets of $\Omega$.  This contradicts Corollary~\ref{cor:vector-space-obstruction}, because $C$ was assumed non-simplicial.  Hence $C$ is simplicial.

Conversely, if $C$ is simplicial, Proposition~\ref{prop:simplicial-RB} shows that $D_C$ is an RB-domain.
\end{proof}

\begin{corollary}\label{cor:FSneqRB}
For every non-simplicial proper cone $C$, the domain $D_C=(-C)\cup\{\bot\}$ is an FS-domain but not an RB-domain.
\end{corollary}

\begin{proof}
The FS statement is Proposition~\ref{prop:FS}.  If $D_C$ were an RB-domain, Theorem~\ref{thm:intro-main} would imply that $C$ is simplicial, contrary to the hypothesis.
\end{proof}

\begin{example}[The Lorentz cone]\label{ex:Lorentz}
Let
\[
        L_3=\{(t,x,y)\in\R^3:t\geq (x^2+y^2)^{1/2}\}.
\]
The section $\{(1,x,y):x^2+y^2\leq1\}$ is a disk, not a simplex.  Hence $L_3$ is not simplicial, and
\[
        (-L_3)_\bot
\]
is an FS-domain but not an RB-domain.  This is precisely the ice-cream-cone case mentioned in \cite[Problem 5.5]{Keimel2009}.
\end{example}

\begin{corollary}[Polyhedral cones]\label{cor:polyhedral}
If $C\subseteq V$ is a polyhedral proper cone, then $D_C$ is an RB-domain if and only if $C$ has exactly $\dim V$ extreme rays.
\end{corollary}

\begin{proof}
A finite-dimensional proper polyhedral cone is simplicial if and only if its extreme rays are generated by a basis of $V$, equivalently if and only if it has exactly $\dim V$ extreme rays.  Apply Theorem~\ref{thm:intro-main}.
\end{proof}

\end{document}